\documentclass[a4paper]{amsart}

\usepackage{latexsym, amssymb,amsmath, amsthm,amsfonts}

\newcommand{\CC}{\mathbb{C}}

\newcommand{\kk}{\Bbbk}

\def\SL{\operatorname{SL}}

\def\SL2{\operatorname{SL}_{2}(K)}

\def\GL2{\operatorname{GL}_{2}(K)}

\def\INVSL2{$K[V]^{operatorname{SL}_{2}(K)}$}
\def\INVSO2{$K[V]^{operatorname{SO}_{2}(K)}$}
\def\INVGL2{$K[V]^{operatorname{GL}_{2}(K)}$}

\def\GL{\operatorname{GL}}
\def\SL{\operatorname{SL}}

\def\N{\mathbb{N}}

\def\Tr{\operatorname{Tr}}

\newtheorem{Lemma}{Lemma}
\newtheorem{Theorem}[Lemma]{Theorem}
\newtheorem{Proposition}[Lemma]{Proposition}
\newtheorem{Corollary}[Lemma]{Corollary}

\newtheorem*{Corollary of Conjecture}{Corollary of Conjecture}

\theoremstyle{definition}

\theoremstyle{remark}
  \newtheorem{rem}[Lemma]{Remark}

\newtheoremstyle{Acknowledgments}
  {}
    {}
     {}
     {}
    {\bfseries}
    {}
     {.5em}
     {\thmname{#1}\thmnumber{ }\thmnote{ (#3)}}
\theoremstyle{Acknowledgments}

\title[Degree bounds for modular covariants]{Degree bounds for modular covariants}

\author{Jonathan Elmer}
\address{Middlesex University\\
The Burroughs, Hendon, London\\
NW4 4BT UK}
\email{j.elmer@mdx.ac.uk}

\author{M\"{u}fit Sezer}
\address{Bilkent University, Department of Mathematics\\
Cankaya, Ankara \\06800 Turkey } \email{sezer@fen.bilkent.edu.tr}

\thanks{The second author is supported by a grant from T\"UBITAK:115F186 }
\date{\today}

\begin{document}
\maketitle
\begin{abstract} Let $V,W$ be representations of a cyclic group $G$ of prime order $p$ over a field $\kk$ of characteristic $p$. The module of covariants $\kk[V,W]^G$ is the set of $G$-equivariant polynomial maps $V \rightarrow W$, and is a module over $\kk[V]^G$. We give a formula for the Noether bound $\beta(\kk[V,W]^G,\kk[V]^G)$, i.e. the minimal degree $d$ such that $\kk[V,W]^G$ is generated over $\kk[V]^G$ by elements of degree at most $d$.
\end{abstract}
\section{Introduction}

Let $G$ be a finite group, $\kk$ a field and $V$, $W$ a pair of finite-dimensional $\kk G$-modules. Let $\kk[V]$ denote the symmetric algebra on the dual $V^*$ of $V$ and let $\kk[V,W] = \kk[V] \otimes_{\kk} W$. Elements of $\kk[V]$ represent polynomial functions $V \rightarrow \kk$ and elements of $\kk[V,W]$ represent polynomial functions $V \rightarrow W$; for $f \otimes w \in \kk[V,W]$ the corresponding function takes $v$ to $f(v)w$. $G$ acts by algebra automorphisms on $\kk[V]$ and hence diagonally on $\kk[V,W]$. The fixed points $\kk[V,W]^G$ of this action are called covariants and represent $G$-equivariant polynomial functions $V \rightarrow W$. The the fixed points $\kk[V]^G$ are called invariants. For $f \in \kk[V]^G$ and $\phi \in \kk[V,W]^G$ we define the product
\[f\phi(v) = f(v)\phi(v).\]
Then $\kk[V]^G$ is a $\kk$-algebra and $\kk[V,W]^G$ is a finite $\kk[V]^G$-module.
Modules of covariants in the non-modular case ($|G| \neq 0 \in \kk$) were studied by Chevalley \cite{Chevalley}, Shephard-Todd \cite{ShephardTodd}, Eagon-Hochster \cite{HochsterEagon}. In the modular case far less is known, but recent work of Broer and Chuai \cite{BroerChuaiRelative} has shed some light on the subject. A systematic attempt to construct generating sets for modules of covariants when $G$ is a cyclic group of order $p$ was begun by the first author in \cite{ElmerCov}.

Let $A = \oplus_{d \geq 0} A_d$ be any graded $\kk$-algebra and $M = \sum_{d \geq 0} M_d$ any graded $A$-module. Then the Noether bound $\beta(A)$ is defined to be the minimum degree $d>0$ such that $A$ is generated by the set $\{a: a \in A_k, k \leq d\}$.  Similarly, $\beta(M,A)$ is defined to be the minimum degree $d>0$ such that $M$ is generated over $A$ by the set $\{m: m \in M_k, k \leq d\}$, and we sometimes write $\beta(M)=\beta(M,A)$ when the context is clear. 

Noether famously showed that $\beta (\CC[V]^G) \leq |G|$ for arbitrary finite $G$, but computing Noether bounds in the modular case is highly nontrivial. When $G$ is cyclic of prime order, the second author along with Fleischmann, Shank and Woodcock \cite{FleischmannSezerWoodcock} determined the Noether bound for any
$\kk G$-module. The purpose of this short article is to find results similar to those in \cite{FleischmannSezerWoodcock} for covariants. Our main result can be stated concisely as:
\begin{Theorem}\label{main}
Let $G$ be a cyclic group of order $p$, $\kk$ a field of characteristic $p$, $V$ a reduced $\kk G$-module  and $W$ a nontrivial indecomposable $\kk G$-module. Then 

\[\beta(\kk[V,W]^G) = \beta(\kk[V]^G)\] unless $V$ is indecomposable of dimension 2.
\end{Theorem}

 \section{Preliminaries} 
For the rest of this article, $G$ denotes a cyclic group of order $p>0$, and we let $\kk$  be a field of characteristic $p$. 
We choose a generator $\sigma$ for $G$. Over $\kk$, there are $p$ indecomposable representations $V_1, \dots ,V_p$ and each indecomposable representation $V_i$ is afforded by a Jordan block of size $i$. Note that $V_p$ is isomorphic to the free module $\kk G$, and this is the unique free indecomposable $\kk G$ -module.
 
 Let $\Delta=\sigma-1\in \kk G$. We define the transfer map $\Tr:\kk [V]\rightarrow \kk [V]$ by $\sum_{1\le i\le p}\sigma^i$. Notice that we also have $\Tr=\Delta^{p-1}$. Invariants that are in the image of $\Tr$ are called transfers.
 
 \begin{rem}\label{repthy}
 Let $e_1, \dots , e_i$ be an upper triangular basis for the $i$-dimensional indecomposable representation $V_i$. The $\Delta (e_j)=e_{j-1}$ for $2\le j\le i$ and $\Delta (e_1)=0$. Therefore $\Delta^j (V_i)=0$ for all $j\ge i$. 
 Note that for an indecomposable module $V_i$  we have $\Delta(V_i)\cong V_{i-1}$ for $2\le i\le p$ and $\Delta (V_1)=0$. It follows that an invariant $f$ is in the image of the linear map $\Delta^j:\kk [V]\rightarrow \kk [V]$ if and only if it is a linear combination of fixed points in indecomposable modules of dimension at least $j+1$.
 In particular, an invariant is in the image of the transfer map ($=\Delta^{p-1}$) if and only if it is a linear combination of fixed points of free $\kk G$-modules.
 \end{rem}
 
 We assume that $V$ and $W$ are $\kk G$-modules with  $W$  indecomposable and we choose a basis $w_1, \dots ,w_n$ for $W$ so that we have 
$$\sigma w_i=\sum_{1\le j\le i} (-1)^{i-j}w_j,$$
for $1\le i\le n$.  For $f\in \kk [V]$ we define the \emph{weight} of $f$ to be the smallest  positive integer $d$ with $\Delta^d (f)=0$. Note that $\Delta^p=(\sigma-1)^p=0$, so the weight of a polynomial is at most $p$.
 
A useful description of covariants is given in \cite{ElmerCov}. We include this description here for completeness.
\begin{Proposition}\cite[Proposition~3]{ElmerCov}
\label{elmer}
Let $f\in \kk [V]$ with weight $d\le n$. Then $$\sum_{1\le j\le d}\Delta ^{j-1}(f)w_j\in \kk [V,W]^G.$$
 Conversely, if 
 $$f_1w_1+f_2w_2+ \cdots +f_nw_n \in \kk[V,W]^G,$$
 then there exists $f\in \kk [V]$ with weight  $\le n$ such that $f_j=\Delta^{j-1 }(f)$ for $1\le j\le n$.
\end{Proposition}
  
For a non-zero covariant $h=f_1w_1+f_2w_2+ \cdots +f_nw_n$, we define the \emph{support} of $h$ to be the largest integer $j$ such that $f_j\neq 0$. We denote the support of $h$ by $s(h)$.
We shall say $h$ is a \emph{transfer covariant} if there exists a non-negative integer $k$ and $f\in \kk[V]$ such that
$f_1=\Delta^k (f)$, $f_2=\Delta^{k+1} (f)$, $\cdots$, $f_{s(h)}=\Delta^{p-1}(f)$ 
 for some $f\in \kk [V]$.
  
We call a homogeneous invariant in $\kk[V]^G$ indecomposable if it is not in the subalgebra of $\kk[V]^G$ generated by invariants of strictly smaller degree. Similarly, a homogeneous covariant in $\kk[V,W]^G$  is indecomposable if it does not lie in the submodule of $\kk[V,W]^G$  generated by covariants of strictly smaller degree. 

\section{Upper bounds}
    We first prove a result on decomposability of a transfer covariant. In the proof below we set $\gamma = \beta(\kk[V],\kk[V]^G)$. 

 \begin{Proposition}
 \label{transfer}
 Let $f\in \kk[V]$  be homogeneous and $h=\Delta^k(f)w_1+\Delta^{k+1}(f)w_2+ \cdots +\Delta^{p-1}(f)w_{s(h)}$ be a transfer covariant of degree $>\gamma$. Then $h$ is decomposable.
\end{Proposition} 
 \begin{proof}
 
 Let  $g_1, \dots , g_t$ be a set of homogeneous polynomials of degree at most $\gamma$ generating   $\kk [V]$ as a module over  $\kk[V]^G$. So we can write $f=\sum_{1\le i\le t} q_ig_i$, where each $q_i\in \kk[V]^G_+$ is a positive degree invariant. Since $\Delta^j$ is  $\kk[V]^G$-linear, we have $\Delta^j(f)=\sum_{1\le i\le t} q_i\Delta^j (g_i)$ for $k\le j\le p-1$. It follows that 
 $$h=\sum_{1\le i\le t}q_i(\Delta^k(g_i)w_1+ \cdots +\Delta^{p-1}(g_i)w_{s(h)}).$$
 Note that $\Delta^k(g_i)w_1+ \cdots +\Delta^{p-1}(g_i)w_{s(h)}$ is a covariant for each $1\le i\le t$ by Proposition \ref{elmer}.  We also have  $q_i\in \kk[V]^G_+$ so it follows that $h$ is decomposable.
 \end{proof}
  Write $V=\oplus_{j=1}^m V_{n_j}$ as a sum of indecomposable modules. Note that $\kk[V\oplus V_1, W]^G=(S(V^*)\otimes S(V_1^*) )\otimes W)^G=\kk [V,W]^G\otimes\kk [V_1]$. Therefore we will assume that $n_j>1$ for all $j$; such representations are called reduced. Choose a basis $\{x_{i,j} \mid 1\le i\le n_j, 1\le j\le m\}$ for $V^*$, with respect to which we have 
  \[\sigma(x_{i,j}) = \left\{\begin{array}{lr} x_{i,j}+ x_{i+1,j} & i<n_j; \\ x_{i,j} & i=n_j. \end{array} \right. \] This induces a multidegree on $\kk[V] = \oplus_{\mathbf{d} \in \N^m} \kk[V]_{\mathbf{d}}$ which is compatible with the action of $G$. For $1\le j\le m$   we define $N_j= \prod_{k=0}^{p-1} \sigma^k x_{1,j}$, and note that the coefficient of $x_{1,j}^p$ in $N_j$ is 1. Given any $f \in \kk[V_{n_j}]$, we can therefore perform long division, writing
 \begin{equation}\label{normdecomp} f = q_jN_j + r\end{equation} where $q_j \in \kk[V_{n_j}]$ for all $j$ and $r \in \kk[V_{n_j}]$ has degree $<p$ in the variable $x_{1,j}$. This induces a vector space decomposition
 \[\kk[V_{n_j}] = N_j\kk[V_{n_j}] \oplus B_j\] where $B_j$ is the subspace of $\kk[V_{n_j}]$ spanned by monomials with $x_{1,j}$-degree $<p$, but the form of the action implies that $B_j$ and its complement are $\kk G$-modules, so we obtain a $\kk G$-module decomposition. 
 Since $\kk[V]=\otimes_{j=1}^m \kk[V_{n_j}]$, it follows that $$\kk[V]=N_j\kk[V] \oplus (B_j\otimes  \kk[V']),$$
 where $V'=V_{n_1}\oplus \cdots  \oplus V_{n_{j-1}}\oplus V_{n_{j+1}} \cdots \oplus V_{n_m}$. From this decomposition it follows that  if $M$ is a $\kk G$ direct summand of $\kk[V]_d$, then $N_jM$ is a $\kk G$ direct summand of $\kk[V]_{d+p}$ with the same isomorphism type. Further, any $f \in \kk[V]^G$ can be written as $f=qN_j+r$ with $q\in \kk[V]^G$ and $r\in (B_j\otimes  \kk[V'])^G $. If in addition $\deg(f) = (d_1, d_2, \ldots, d_m)$ with $d_j>p-n_j$, then the degree $d_j$ homogeneous component of $B_j$ is free by  \cite[2.10]{HughesKemper} and since tensoring a module with a free (projective) module gives a free (projective) module we may further assume, by Remark \ref{repthy}, that $r$ is in the  image of the transfer  map. 
  
If $h = \sum_{i=1}^{s(h)} \Delta^{i-1}(f)w_i \in \kk[V,W]^G$, we define the multidegree of $h$ to be that of $f$. Since $G$ preserves the multidegree, this is the same as the multidegree of $\Delta^{i-1}(f)$ for all $i \leq s(h)$. Then the analogue of this result for covariants is the following:  
 \begin{Proposition}
 \label{sum}
 Let $h$ be a covariant of multidegree $d_1,d_2, \ldots, d_m$ with $d_j>p-n_j$ for some $j$. Then there exists a covariant $h_1$ and a transfer covariant  $h_2$ such that $h=N_jh_1+h_2$.
 \end{Proposition}
  \begin{proof}
 We proceed by induction on the support $s(h)$ of $h$. If $s(h)=1$, then by Proposition \ref{elmer}, we have that $h=fw_1$ with $f\in \kk [V]^G$. Then we can write $f=qN_j+\Delta^{p-1}(t)$ for some $q\in \kk[V]^G$ and $t \in \kk[V]$. Then both $qw_1$ and $\Delta^{p-1}(t)w_1$ are covariants by Proposition \ref{elmer} and therefore $h=qN_jw_1+\Delta^{p-1}(t)w_1$ gives us the desired decomposition. 
 
 Now assume that  $s(h)=k$. Then by Proposition \ref{elmer} there exists $f\in \kk [V]$ such that $$h=fw_1+\Delta (f)w_2+\cdots +\Delta ^{k-1}(f)w_k,$$
 with $\Delta^k(f)=0$. Since $\Delta^{k-1}(f) \in \kk[V]^G$ and $d_j>p-n_j$, we can write $\Delta^{k-1}(f)=qN_j+\Delta^{p-1}(t)$ for some $q\in \kk[V]^G$ and $t \in \kk[V]$. It follows that $qN_j$ is in the image of $\Delta^{k-1}$. But since multiplication by $N_j$ preserves the isomorphism type of a module, it  follows that $q$ is in the image of $\Delta^{k-1}$. Write $q=\Delta^{k-1}(f')$ with  $f'\in \kk[V]$.
    Set $h_1=f'w_1+\Delta (f')w_2+\cdots +\Delta ^{k-1}(f')w_k$ and  $h_2=\Delta^{p-k}(t)w_1+\cdots +\Delta^{p-1}(t)w_k$. Since $\Delta^{k-1}(f')\in \kk[V]^G$, $h_1$ is a  covariant by Proposition  
 \ref{elmer}. Consider the covariant $h'=h-N_jh_1-h_2$. Since $\Delta^{k-1}(f)=\Delta^{p-1}(t)+\Delta ^{k-1}(f')N_j$, the support of $h'$ is strictly smaller than the support of $h$. Moreover, $h_2$ is a transfer covariant and so  the assertion of the proposition follows by induction. 
 \end{proof}
 
 We obtain the following upper bound for the Noether number of covariants:
 \begin{Proposition}\label{ub}
 $\beta(\kk[V,W]^G) \leq \max(\beta(\kk[V],\kk[V]^G), mp-\dim(V))$.
 \end{Proposition}
 
 \begin{proof} Let $h \in \kk[V,W]^G$ with degree $d> \max(\beta(\kk[V],\kk[V]^G), mp-\dim(V))$. Let $(d_1, d_2, \ldots, d_m)$ be the multidegree of $h$. Then we must have $d_j>p-n_j$ for some $j$. Consequently we may apply Proposition \ref{sum}, writing 
 $$h = N_j h_1 + h_2$$ where $h_2$ is a transfer covariant. Since $\deg(h_2)> \beta(\kk[V], \kk[V]^G)$, $h_2$ is decomposable by Proposition \ref{transfer}, and so we have shown that $h$ is decomposable.
 \end{proof}
 
 \section{Lower bounds}
 
 Indecomposable transfers are one method of obtaining lower bounds for $\beta(\kk[V]^G)$. The analogous result for covariants is:
 \begin{Lemma}\label{lb}
 Let $n \ge 2$ and let $\Delta^{p-1}(f)\in \kk [V]^G$ be an indecomposable homogeneous transfer. Then the transfer covariant $$h=\Delta^{p-n}(f)w_1+\cdots +\Delta^{p-1}(f)w_n$$
 is indecomposable.
 \end{Lemma}
 \begin{proof}
 Assume on the contrary that $h$ is decomposable. Then  there exist homogeneous $q_i\in \kk[V]^G_+ $
and $h_i\in \kk[V,W]^G$   such that $h=\sum_{1\le i\le t}q_ih_i$. Write $h_i=h_{i,1}w_1+\cdots +h_{i,n}w_n$ for $1\le i\le t$. Then we have $\Delta^{p-1}(f)=\sum_{1\le i\le t}q_ih_{i,n}$. By Proposition \ref{elmer} we have  $\Delta (h_{i,n-1})=h_{i,n}$ and so $h_{i,n}\in \kk[V]^G_+$ because $n\ge 2$. It follows that  $\sum_{1\le i\le t}q_ih_{i,n}$ is a decomposition of $\Delta^{p-1}(f)$  in terms of invariants of strictly smaller degree, contradicting the indecomposability of $\Delta^{p-1}(f)$.\end{proof}
 
 \begin{Corollary}\label{lbcor}
 Suppose $n \geq 2$ and $\beta(\kk[V]^G) > \max(p,mp-\dim(V))$. Then $\beta(\kk[V]^G) \leq \beta(\kk[V,W]^G)$.
 \end{Corollary}
 
 \begin{proof} By \cite[Lemma~2.12]{HughesKemper}, $\kk[V]^G$ is generated by the norms $N_1,N_2, \ldots, N_m$, invariants of degree at most $mp-\dim(V)$, and transfers. Since there exists an indecomposable invariant of degree $\beta(\kk[V]^G)$, if the hypotheses of the corollary above hold, then $\kk[V]^G$ contains an indecomposable transfer with this degree. By Lemma \ref{lb}, $\kk[V,W]^G$ contains a transfer covariant of degree $\beta(\kk[V]^G)$ which is indecomposable, from which the conclusion follows.
 \end{proof}
 
 \section{Main results}
 
We are now ready to prove Theorem \ref{main}. Note that $\kk[V,V_1]^G$ is generated over $\kk[V]^G$ by $w_1$ alone, which has degree zero, and therefore $\beta(\kk[V,V_1]^G) = 0$. For this reason we assume $n \geq 2$ throughout.
\begin{proof}
Suppose first that $n_j>3$ for some $j$. Then by \cite[Proposition~1.1(a)]{FleischmannSezerWoodcock}, we have
\[\beta(\kk[V]^G) = m(p-1)+(p-2).\] Since $V$ is reduced we have $\dim(V) \geq 2m$ and hence
\[\beta(\kk[V]^G) > m(p-2) \geq mp-\dim(V).\]
Also, $\beta(\kk[V]^G) \geq 2p-3 >p$ since $n_j \leq p$ for all $j$. Therefore Corollary \ref{lbcor} implies that $\beta(\kk[V]^G) \leq \beta(\kk[V,W]^G)$.
On the other hand, \cite[Lemma~3.3]{FleischmannSezerWoodcock} shows that the top degree of $\kk[V]/\kk[V]^G_+\kk[V]$ is bounded above by $m(p-1)+(p-2)$. By the graded Nakayama Lemma it follows that $\beta(\kk[V],\kk[V]^G) \leq m(p-1)+(p-2)$. We have already shown that this number is at least $mp-\dim(V)+1$, so by Proposition \ref{ub} we get that
\[\beta(\kk[V,W]^G) \leq m(p-1)+(p-2) = \beta(\kk[V]^G)\] as required.

Now suppose that $n_i \leq 3$ for all $i$ and $n_j=3$ for some $j$. Then by \cite[Proposition~1.1(b)]{FleischmannSezerWoodcock}, we have
\[\beta(\kk[V]^G) = m(p-1)+1.\] Since $V$ is reduced we have $\dim(V) \geq 2m$ and hence
\[\beta(\kk[V]^G) > m(p-2) \geq mp-\dim(V).\] Also $\beta(\kk[V]^G) \geq 2p-1>p$ provided $m \geq 2$. In that case Corollary \ref{lbcor} applies. If $m=1$ then Dickson \cite{DicksonMadison} has shown that $\kk[V]^G=\kk[x_1, x_2, x_3]^G$  is minimally generated by the invariants $x_3$, $x_2^2-2x_1x_3-x_2x_3$, $N$, $\Delta^{p-1}(x_1^{p-1}x_2)$. It follows that $\Delta^{p-1}(x_1^{p-1}x_2)$ is an indecomposable transfer, so by Lemma \ref{lb},  $\kk[V,W]^G$ contains an indecomposable transfer covariant of degree $p = \beta(\kk[V]^G)$. In either case we obtain
\[\beta(\kk[V,W]^G) \geq \beta(\kk[V]^G).\]

On the other hand, by \cite[Corollary~2.8]{SezerShankCoinv}, $m(p-1)+1$ is an upper bound for the top degree of $\kk[V]/\kk[V]^G_+$. By the same argument as before we get $\beta(\kk[V]^G,\kk[V]) \leq m(p-1)+1$. We have already shown that this number is at least $mp-\dim(V)+1$, so by Proposition \ref{ub} we get that
\[\beta(\kk[V,W]^G) \leq m(p-1)+1 = \beta(\kk[V]^G)\] as required.

It remains to deal with the case $n_i=2$ for all $i$, i.e. $V = mV_2$. We assume $m \geq 2$. In this case Campbell and Hughes \cite{CampbellHughes} showed that $\beta(\kk[V]^G) = (p-1)m$. As $\dim(V) =2m$ we have $\beta(\kk[V]^G)>m(p-2) = mp-\dim(V)$. If $m \geq 3$ or $m = 2$ and $p>2$ then we have
$\beta(\kk[V]^G)>p$ and Corollary \ref{lbcor} applies. In case $m=2=p$, $\kk[V]^G = \kk[x_{1,1},x_{2,1},x_{1,2},x_{2,2}]^G$ is a hypersurface, minimally generated by $\{x_{2,1},N_1,x_{2,2},N_2, \Delta^{p-1}(x_{1,1}x_{1,2})\}$. In particular $ \Delta^{p-1}(x_{1,1}x_{1,2})$ is an indecomposable transfer, so by Lemma \ref{lb}, $\kk[V,W]^G$ contains an indecomposable transfer covariant of degree 2. In both cases we get \[\beta(\kk[V,W]^G) \geq \beta(\kk[V]^G).\]
On the other hand, by \cite[Theorem~2.1]{SezerShankCoinv}, the top degree of $\kk[V]/\kk[V]^G_+ \kk[V]$ is bounded above by $m(p-1)$. We have already shown this number is at least $mp-\dim(V)+1$. Therefore by Proposition \ref{ub} we get $\beta(\kk[V,W]^G) \leq \beta(\kk[V]^G)$ as required.
\end{proof}
 
 \begin{rem} The only reduced representation not covered by Theorem \ref{main} is $V=V_2$. An explicit minimal set of generators of $\kk[V_2,W]^G$ as a module over $\kk[V_2]^G$ is given in \cite{ElmerCov}, the result is $$\beta(\kk[V_2,W]) = n-1.$$ This is the only situation in which the Noether number is seen to depend on $W$.
 \end{rem}
 
 \begin{rem} Suppose $V$ is any reduced $\kk G$-module and $W = \bigoplus_{i=1}^r W_i$ is a decomposable $\kk G$-module. Then
 \[\kk[V,W]^G = ((S(V^*)\otimes (\oplus_{i=1}^r W_i))^G = \bigoplus_{i=1}^r (S(V^*) \otimes W_i)^G.\]
 So $\beta(\kk[V,W]^G) = \max\{(\beta(\kk[V,W_i]^G): i=1, \ldots, r)\} = \beta(\kk[V]^G)$ unless $V$ is indecomposable of dimension 2, in which case we have \[\beta(\kk[V_2,W]^G) = \max\{(\beta(\kk[V_2,W_i]^G): i=1, \ldots, r)\} = \max\{\dim(W_i)-1: i=1, \ldots, r\}. \] 
 Thus, the results of this paper can be used to compute $\beta(\kk[V,W]^G)$ for arbitrary $\kk G$-modules $V$ and $W$.
 
 \end{rem}

\bibliographystyle{plain}
\bibliography{MyBib}

\end{document}